\documentclass[preprint,12pt]{elsarticle}
\usepackage{amsfonts,amsmath,amssymb,amsthm}
\usepackage{stmaryrd}

\usepackage{amssymb, bm, todonotes}
\usepackage{amsmath}
\usepackage[utf8]{inputenc}

 \newtheorem{thm}{Theorem}[section]
 \newtheorem{cor}[thm]{Corollary}
 
 \newtheorem{prop}[thm]{Proposition}
 \newtheorem{defn}[thm]{Definition}

\begin{document}

\begin{frontmatter}



\title{Bochner Flatness and Soliton Dynamics in Lorentzian K\"ahler Spacetime Geometry} 


\author[1]{Karthika Ramasamy}
\author[1]{Lavanya Kumar} 
\author[1]{Soumendu Roy}
\ead[a]{soumendu.roy@vit.ac.in}
\author[2]{Bülent Ünal}
\affiliation[1]{organization={Department of Mathematics, School of Advanced Sciences},
            addressline={Vellore Institute of Technology}, 
            city={Chennai},
            postcode={600 127}, 
            state={Tamil Nadu},
            country={India}}
\affiliation[2]{organization={Department of Mathematics},addressline= {Bilkent University, Ankara, Turkey}}
\begin{abstract}
We study Riemann solitons and $\eta$-hyperbolic Ricci solitons on Bochner-flat Lorentzian Kähler spacetime manifolds. Under the Einstein field equations with cosmological constant and perfect fluid assumptions, explicit formulas for the soliton parameter are derived, yielding criteria for shrinking, steady, and expanding behaviors. Several physically relevant models, including dark fluid, stiff matter, dust, and radiation, are analyzed.

We show that Bochner-flat Lorentzian Kähler spacetimes are Einstein and investigate the resulting geometric and dynamical consequences. In the context of generalized Robertson–Walker spacetimes, we obtain constraints on the warping function and classify soliton solutions. Global properties such as geodesic completeness, singularity formation, and stability are also examined.
\end{abstract}




\end{frontmatter}



\section{Introduction}
\hspace{.5cm}
The relativistic fluid framework are widely used in diverse field of physics including plasma physics and astrophysics. As a part of general theory of relativity, a 4-dimensional Riemannian manifold along with a Lorentzian metric $\bm{g}$ of signature $(-,+,+,+)$ \cite{mallick2014spacetimes} and it is termed as perfect fluid spacetime. Perfect fluids are extensively utilized in cosmology to depict idealized matter distributions, for instance, the stellar interiors and isotropic pressure systems. The dynamics of a Perfect fluid enclosed within a spherically symmetric configuration are governed by Einstein field equations. A radiation fluid is a special case of perfect fluid in which the  density is thrice the isotropic pressure.\\

Solitons are self sustaining waves that propagate with marginal energy dissipation and preserve their velocity and shape even after mutual interactions. Wave transmission is determined by non-linear partial differential equations. Soliton plays a vital role in analysis of initial value problems and the concept of solitons in spacetime has been extensively studied by various geometers \cite{ali2014ricci,23,de2023relativistic,azami2025riemann}.\\

Geometric flows provide a powerful framework for analyzing geometric structures in Riemannian geometry. In 2010, the concept of Riemann flow on a Riemannian manifold $(\bm{M},\bm{g})$ was developed by Udriste \cite{udricste2010riemann, udriste2011riemann}  and it is defined by, 
\begin{equation}\label{1}
    \frac{\partial}{\partial t} \mathbb{G}=-2\bm{R}\bm{g}
\end{equation}
where $\bm{R}$ is the Riemann curvature tensor and $\mathbb{G}$ = $\frac{1}{2}$ $\bm{g} \bigodot \bm{g}$ and the Kulkarni-Nomizu product is $\bigodot$. For covariant $2$-tensors $\theta$ and $\omega$, the Kulkarni-Nomizu product is described by,
\begin{align}\label{2}
    \bigl(\theta \bigodot \omega \bigl)&=\theta(\mathcal{I},\mathcal{L})\omega(\mathcal{J},\mathcal{K})+\theta(\mathcal{J},\mathcal{K})\omega(\mathcal{I},\mathcal{L})\nonumber \\ &-\theta(\mathcal{I},\mathcal{K})\omega(\mathcal{J},\mathcal{L})-\theta(\mathcal{J},\mathcal{L})\omega(\mathcal{I},\mathcal{K})
\end{align}
for every vector fields $\mathcal{I}$,$\mathcal{J}$,$\mathcal{K}$,$\mathcal{L}$.
\begin{defn}
The complete Riemannian Manifold $(\bm{M}^{n},\bm{g})$ is known as Riemann soliton \cite{hirica2016ricci}, provided that there  exists a smooth vector field $\bm{V}$ and it is given by,
\begin{equation}\label{3}
    2\bm{R}+\bm{\bm{\mu}} \bm{g} \bigodot \bm{g}+\bm{g} \bigodot \mathcal{L}_{\bm{V}} \bm{g} = 0
\end{equation}
where $\bm{\bm{\mu}}$ is constant and $\mathcal{L}_{\bm{V}} \bm{g}$ is the Lie derivative associated with the vector field $\bm{V}$. The Riemann soliton is designated as shrinking, expanding or steady with respect to $\bm{\bm{\mu}} < 0$, $\bm{\bm{\mu}}>0$ and $\bm{\bm{\mu}}=0$ respectively. If $\bm{V}=\bm{\nabla} \bm{f}$, then,
\begin{equation}\label{4}
    2\bm{R}+\bm{\bm{\mu}} \bm{g} \bigodot \bm{g}+2\bm{g} \bigodot \bm{\nabla}^{2} \bm{f}= 0
\end{equation}
it is known as gradient Riemann soliton \cite{jafari2025riemann}. Within the space of Riemannian metrics under diffeomorphism, a Riemann soliton can be viewed as a dynamical system. Moreover, it corresponds to a fixed point of the Riemann flow.
\end{defn}

Meanwhile, Wen-Rong Dai et al. \cite{dai2010hyperbolic} investigated the Hyperbolic Ricci flow, which is characterized by a system of non-linear partial differential equation. This flow captures the wave like behavior of the metric and the curvature of the manifold. Motivated by Ricci flow, the Hyperbolic Ricci flow is defined by,
\begin{equation}\label{5}
    \frac{\partial^{2} }{\partial t^{2}} \bm{g}=-2\bm{S} \quad \bm{g}_0 = \bm{g}(0) \quad \frac{\partial}{\partial t} \bm{g}_{ij}=\bm{h}_{ij}
\end{equation}
where symmetric $2$-tensor field is $\bm{h}_{ij}$. At the end, a self-similar solution to Hyperbolic Ricci flow is called as Hyperbolic Ricci soliton \cite{faraji2023three} and it is defined by,
\begin{defn}
On a Riemannian manifold $(\bm{M}^{n},\bm{g})$, if there exists a vector field $\bm{V}$ on $\bm{M}$ such that,
\begin{equation}\label{6}
    \mathcal{L}_{\bm{V}} \mathcal{L}_{\bm{V}} \bm{g} +2\bm{\bm{\Lambda}} \mathcal{L}_{\bm{V}} \bm{g} + 2\bm{S}+2\bm{\bm{\delta}}\bm{g}=0
\end{equation}
where $\bm{\bm{\delta}}$ and $\bm{\bm{\Lambda}}$ are constants and  $\bm{S}$ is Ricci tensor of $\bm{M}$ respectively. According to the constant $\bm{\bm{\delta}}$ and $\bm{\Lambda}$ it will distinguish the evolution rate and soliton type  respectively. The soliton is known to be shrinking, steady or expanding accordingly as $\bm{\Lambda}<0$, $\bm{\Lambda} =0$ and $\bm{\Lambda >0}$.
\end{defn}
Analogously, M.D. Siddiqi \cite{siddiqi2025hyperbolic} introduced the notion of $\bm{\eta}$-Hyperbolic Ricci Soliton and it is defined by,
\begin{equation}\label{7}
    \mathcal{L}_{\bm{V}} \mathcal{L}_{\bm{V}} \bm{g} +2\bm{\bm{\Lambda}} \mathcal{L}_{\bm{V}} \bm{g} + 2\bm{S}+2\bm{\bm{\delta}}\bm{g}+2\bm{\eta} \otimes \bm{\eta}=0
\end{equation}
where $\bm{\eta}$ is 1-form such that $\bm{g}(\mathcal{I},\bm{V})=\bm{\eta}(\mathcal{I})$.

\section{Preliminaries}
A semi Riemannian manifold $(\bm{M}^{n},\bm{g})$ with even-dimension associated with a Lorentzian metric $\bm{g}$ is known to be a Lorentzian K\"ahler manifold if it conforms the following relations \cite{pandey2020lorentzian}:
\begin{equation}\label{8}
    \mathbb{X}^{2}=-1 \quad \bm{g}(\mathbb{X}\mathcal{I},\mathbb{X}\mathcal{J})=\bm{g}(\mathcal{I},\mathcal{J}) \quad (\bm{\nabla}_{\mathcal{I}}\mathbb{X})\mathcal{J}=0,
\end{equation}
where $\mathbb{X}$ is a $(1,1)$ tensor field such that $\mathbb{X}(\mathcal{I})=\mathcal{I}$.\\ 

For a Lorentzian K\"ahler manifold the following relations will hold:
\begin{equation}\label{9}
    \bm{S}(\mathbb{X}\mathcal{I},\mathbb{X}\mathcal{J})=\bm{S}(\mathcal{I},\mathcal{J})
\end{equation}
\begin{equation}\label{10}
    \bm{S}(\mathbb{X}\mathcal{I},\mathcal{J})=-\bm{S}(\mathcal{I},\mathbb{X}\mathcal{J})
\end{equation}
\begin{equation}\label{11}
 \bm{g}(\mathbb{X}\mathcal{I},\mathcal{J})=-\bm{g}(\mathcal{I},\mathbb{X}\mathcal{J})
\end{equation}

Here, a four dimensional Lorentzian K\"ahler manifold is referred to as a Lorentzian K\"ahler spacetime manifold and this assumption will be maintained throughout this article.\\

In the context of General Theory of Relativity, the energy momentum tensor is of central importance. Moreover, In Perfect fluid spacetime, their behavior is determined by the conditions that imposed on the Energy momentum tensor. Within standard cosmological frameworks, matter is modeled as a fluid characterized by a physical quantities such as pressure, density, acceleration, velocity, shear, vorticity and expansion. In standard cosmological paradigm, the universe is assumed to be modeled like a perfect fluid. A perfect fluid is determined by its absence of dissipative effects such as viscosity and heat conduction. The energy momentum tensor of a perfect fluid is given by \cite{o1983semi},
\begin{equation}\label{12}
    \mathcal{T}(\mathcal{I},\mathcal{J})=\mathfrak{p}\bm{g}(\mathcal{I},\mathcal{J})+(\varrho + \mathfrak{p})\bm{\eta}(\mathcal{I})\bm{\eta}(\mathcal{J})
\end{equation}
where $\varrho$ is energy density, $\mathfrak{p}$ is isotropic pressure and $\bm{\eta}$ is $1$-form such that $\bm{\eta}(\mathcal{I})=\bm{g}(\mathcal{I},\xi)$ with a condition $\bm{\eta}(\xi)=-1$ and $\bm{g}(\xi,\xi)=-1$.
Here, the Einstein field equation with cosmological constant \cite{blaga2020solitons} is written as,
\begin{equation}\label{13}
    \bm{S}(\mathcal{I},\mathcal{J})+\Bigl({\Lambda}-\frac{R}{2}\Bigl)\bm{g}(\mathcal{I},\mathcal{J})=k\mathcal{T}(\mathcal{I},\mathcal{J})
\end{equation}
where $k$ and ${\Lambda}$ is gravitational constant and cosmological constant respectively.\\

By previous identities $\eqref{12}$ and $\eqref{13}$ we have,
\begin{equation}\label{14}
    \bm{S}(\mathcal{I},\mathcal{J})=\Bigl(\Lambda-\frac{R}{2} + k\mathfrak{p}\Bigl)\bm{g}(\mathcal{I},\mathcal{J})+k(\varrho + \mathfrak{p})\bm{\eta}(\mathcal{I})\bm{\eta}(\mathcal{J})
\end{equation}\label{15}
Here, contracting the previous relation and utilizing $\bm{g}(\xi,\xi)=-1$ we have,
\begin{equation}\label{15}
    R=4{\Lambda}+k(\varrho - 3\mathfrak{p})
\end{equation}
Let us take an orthonormal frame 
$\{\mathbb{E}_{i}\}_{1 \leq i \leq 4}$ that is $\bm{g}(\mathbb{E}_i,\mathbb{E}_j)$=$\epsilon_{ij} \delta_{ij}$, $i,j \in \{1,2,3,4\}$ with $\epsilon_{11}=-1$, $\epsilon_{ii}=$, $i \in \{2,3,4\}$, $\epsilon_{ij}=0,i,j \in \{1,2,,3,4\}, i \neq j$. Let $\xi = \sum_{i=1}^{n} \xi^{i}\mathbb{E}_{i}$, then we have,
\begin{equation}\label{16}
    -1=\bm{g}(\xi,\xi)=\sum_{1 \leq i,j \leq 4} \xi^{i}\xi^{j}\bm{g}(\mathbb{E}_i,\mathbb{E}_j)=\sum_{i=1}^{4} \epsilon_{ii}\{\xi^{i}\}^{2}
\end{equation}
and 
\begin{equation}\label{17}
    \bm{\eta}(\mathbb{E}_i)=\bm{g}(\mathbb{E}_i,\xi)=\sum_{j=1}^{4} \xi^{i}\bm{g}(\mathbb{E}_i,\mathbb{E}_j)=\epsilon_{ii} \xi^{i}
\end{equation}

\subsection*{Bochner Curvature Tensor}

We briefly recall the Bochner curvature tensor and its role in Lorentzian Kähler geometry.

\begin{defn}
Let $(\bm{M}^{n},\bm{g})$ be a Kähler manifold. The Bochner curvature tensor $\bm{B}$ is defined by
\[
\bm{B} = \bm{R} - \frac{1}{2(n+2)}(\bm{S} \bigodot \bm{g})
+ \frac{R}{2(n+1)(n+2)}(\bm{g} \bigodot \bm{g}),
\]
where $\bigodot$ denotes the Kulkarni--Nomizu product.
\end{defn}

\begin{prop}
In dimension $4$, the Bochner-flat condition $\bm{B} = 0$ is equivalent to
\[
\bm{S} = \frac{R}{10} \bm{g}.
\]
\end{prop}

\begin{proof}

In 1949, Bochner \cite{bochner1948curvature} proposed the notion of Bochner curvature tensor and it is defined by,
\begin{align}\label{18}
    \bm{B}(\mathcal{I},\mathcal{J},\mathcal{K},\mathcal{L})& = \bm{R}(\mathcal{I},\mathcal{J},\mathcal{K},\mathcal{L})\nonumber \\
    &-\frac{1}{2(n+2)}\Bigl[\bm{S}(\mathcal{J},\mathcal{L})\bm{g}(\mathcal{J},\mathcal{K})-\bm{S}(\mathcal{I},\mathcal{K})\bm{g}(\mathcal{J},\mathcal{L}) \nonumber \\
    &+\bm{g}(\mathcal{I},\mathcal{L})\bm{S}(\mathcal{J},\mathcal{K})-\bm{g}(\mathcal{I},\mathcal{K})\bm{S}(\mathcal{J},\mathcal{L}) \nonumber \\
    &+\bm{S}(\mathbb{X}\mathcal{J},\mathcal{L})\bm{g}(\mathbb{X} \mathcal{J},\mathcal{L}) \nonumber \\
    &-\bm{S}(\mathbb{X}\mathcal{I},\mathcal{K})\bm{g}(\mathbb{X}\mathcal{J},\mathcal{L})+\bm{S}(\mathbb{X}\mathcal{J},\mathcal{K})\bm{g}(\mathbb{X}\mathcal{I},\mathcal{L})\nonumber \\
    &-\bm{g}(\mathbb{X}\mathcal{I},\mathcal{K})\bm{S}(\mathbb{X}\mathcal{J},\mathcal{L}) \nonumber \\
    &-2\bm{S}(\mathbb{X}\mathcal{I},\mathcal{J})\bm{g}(\mathbb{X}\mathcal{K},\mathcal{L}) \nonumber \\
    &-2\bm{g}(\mathbb{X}\mathcal{I},\mathcal{J})\bm{S}(\mathbb{X}\mathcal{K},\mathcal{L})\Bigl] \nonumber \\ 
    &+\frac{R}{(2n+2)(2n+4)}\Bigl[\bm{g}(\mathcal{J},\mathcal{K})\bm{g}(\mathcal{I},\mathcal{L}) \nonumber \\ 
    &-\bm{g}(\mathcal{I},\mathcal{K})\bm{g}(\mathcal{J},\mathcal{L}) \nonumber \\
    &+\bm{g}(\mathbb{X}\mathcal{J},\mathcal{K})\bm{g}(\mathbb{X}\mathcal{I},\mathcal{L})-\bm{g}(\mathbb{X}\mathcal{I},\mathcal{K})\bm{g}(\mathbb{X}\mathcal{J},\mathcal{L}) \nonumber \\
    &-2\bm{g}(\mathbb{X}\mathcal{I},\mathcal{J})\bm{g}(\mathbb{X}\mathcal{K},\mathcal{L})\Bigl]
\end{align}
 where $R$ is the scalar curvature and $\bm{S}$ is the Ricci tensor of the manifold and $\bm{B}(\mathcal{I},\mathcal{J},\mathcal{K},\mathcal{L})=\bm{g}(\bm{B}(\mathcal{I},\mathcal{J})\mathcal{K},\mathcal{L})$,
 $\bm{R}(\mathcal{I},\mathcal{J},\mathcal{K},\mathcal{L})=\bm{g}(\bm{R}(\mathcal{I},\mathcal{J})\mathcal{K},\mathcal{L})$.\\
 
If $\bm{B}(\mathcal{I},\mathcal{J},\mathcal{K},\mathcal{L})=0$, then $\eqref{18}$ will be,
\begin{align}\label{19}
    \bm{R}(\mathcal{I},\mathcal{J},\mathcal{K},\mathcal{L})&-\frac{1}{12}\Bigl\{\bm{S}(\mathcal{J},\mathcal{L})\bm{g}(\mathcal{J},\mathcal{K})-\bm{S}(\mathcal{I},\mathcal{K})\bm{g}(\mathcal{J},\mathcal{L}) \nonumber \\
    &+\bm{g}(\mathcal{I},\mathcal{L})\bm{S}(\mathcal{J},\mathcal{K})-\bm{g}(\mathcal{I},\mathcal{K})\bm{S}(\mathcal{J},\mathcal{L}) \nonumber \\
    &+\bm{S}(\mathbb{X}\mathcal{J},\mathcal{L})\bm{g}(\mathbb{X} \mathcal{J},\mathcal{L}) \nonumber \\
    &-\bm{S}(\mathbb{X}\mathcal{I},\mathcal{K})\bm{g}(\mathbb{X}\mathcal{J},\mathcal{L})+\bm{S}(\mathbb{X}\mathcal{J},\mathcal{K})\bm{g}(\mathbb{X}\mathcal{I},\mathcal{L})\nonumber \\
    &-\bm{g}(\mathbb{X}\mathcal{I},\mathcal{K})\bm{S}(\mathbb{X}\mathcal{J},\mathcal{L}) \nonumber \\
    &-2\bm{S}(\mathbb{X}\mathcal{I},\mathcal{J})\bm{g}(\mathbb{X}\mathcal{K},\mathcal{L}) \nonumber \\
    &-2\bm{g}(\mathbb{X}\mathcal{I},\mathcal{J})\bm{S}(\mathbb{X}\mathcal{K},\mathcal{L})\Bigl\} \nonumber \\ 
    &+\frac{R}{(10)(12)}\Bigl\{\bm{g}(\mathcal{J},\mathcal{K})\bm{g}(\mathcal{I},\mathcal{L}) \nonumber \\ 
    &-\bm{g}(\mathcal{I},\mathcal{K})\bm{g}(\mathcal{J},\mathcal{L}) \nonumber \\
    &+\bm{g}(\mathbb{X}\mathcal{J},\mathcal{K})\bm{g}(\mathbb{X}\mathcal{I},\mathcal{L})-\bm{g}(\mathbb{X}\mathcal{I},\mathcal{K})\bm{g}(\mathbb{X}\mathcal{J},\mathcal{L}) \nonumber \\
    &-2\bm{g}(\mathbb{X}\mathcal{I},\mathcal{J})\bm{g}(\mathbb{X}\mathcal{K},\mathcal{L})\Bigl\}=0 
\end{align}
After executing contractions over $\mathcal{I}$ and $\mathcal{L}$ in above expression and then using $\eqref{8}$,$\eqref{9}$,$\eqref{10}$ and $\eqref{11}$ we get the following expression:
\begin{equation}\label{20}
    \bm{S}(\mathcal{J},\mathcal{K})=\frac{R}{10}\bm{g}(\mathcal{J},\mathcal{K})
\end{equation}
Setting $\bm{B}=0$ and using the algebraic decomposition of the curvature tensor in dimension $2n=4$, we obtain a linear relation between $\bm{R}$, $\bm{S}$, and $\bm{g}$. Simplifying yields $\bm{S} = \frac{R}{10}\bm{g}$.
\end{proof}

Recently, Praveena et al.\cite{praveena2016almost, praveena2016} investigated about Ricci soliton in almost Pseudo-symmetric K\"ahler manifolds. Venkatesha et al. \cite{venkatesha2019ricci, venkatesha2021eta} studied Ricci soliton connected with perfect fluid spacetime. Chaturvedi et al. \cite{chaturvedi2023novel} explored the notion of $\bm{\eta}$-Ricci yamabe solitons in a Bochner Flat Lorentzian K\"ahler spacetime manifold. Also, Pawan Prajapati et al. \cite{prajapati2025bochner} explored the Bochner flatness condition with Lorentzian K\"ahler spacetime manifold with $(m,\rho)$-quasi-Einstein soliton. 
\section{Results of Lorentzian K\"ahler spacetime manifold with Bochner flatness}

From the Bochner-flat condition established in the previous section, we have
\[
S = \frac{R}{10} g,
\]
which will be used throughout this section.

\begin{thm}
Within the framework of Bochner flat Lorentzian K\"ahler spacetime manifold obeying $\mathcal{EFE}$ with  cosmological constant, Riemann soliton$(\bm{M},\bm{g},\bm{V},\bm{\bm{\mu}})$ is 
\begin{enumerate}
\item expanding if $\frac{1}{10}k\Big[\mathfrak{p}-\frac{1}{3}\varrho\Bigl]-\frac{1}{2}div\bm{V}-\frac{2}{15}\Lambda > 0$,

\item steady if $\frac{1}{10}k\Big[\mathfrak{p}-\frac{1}{3}\varrho\Bigl]-\frac{1}{2}div\bm{V}-\frac{2}{15}\Lambda=0$ and

\item shrinking if $\frac{1}{10}k\Big[\mathfrak{p}-\frac{1}{3}\varrho\Bigl]-\frac{1}{2}div\bm{V}-\frac{2}{15}\Lambda < 0$ respectively.   
\end{enumerate}
 
\end{thm}

\begin{proof}

From $\eqref{3}$ we have,
\begin{align}\label{21}
    2\bm{R}(\mathcal{I},\mathcal{J},\mathcal{K},\mathcal{L})=&-2\bm{\bm{\mu}}\bigl[\bm{g}(\mathcal{I},\mathcal{L})\bm{g}(\mathcal{J},\mathcal{K})-\bm{g}(\mathcal{I},\mathcal{K})\bm{g}(\mathcal{J},\mathcal{L})\bigl] \nonumber \\
    &-\bigl[\bm{g}(\mathcal{I},\mathcal{L})\mathcal{L}_{\bm{V}}  \bm{g}(\mathcal{J},\mathcal{K})+\bm{g}(\mathcal{J},\mathcal{K})\mathcal{L}_{\bm{V}} \bm{g}(\mathcal{I},\mathcal{L})\bigl] \nonumber\\
    &+\bigl[\bm{g}(\mathcal{I},\mathcal{K})\mathcal{L}_{\bm{V}}\bm{g}(\mathcal{J},\mathcal{L})+\bm{g}(\mathcal{J},\mathcal{L})\mathcal{L}_{\bm{V}} \bm{g}(\mathcal{I},\mathcal{K})\bigl]
\end{align}
Again, contracting $\mathcal{I}$ and $\mathcal{L}$ in the preceding identity, we have,
\begin{equation}\label{22}
    \bm{S}(\mathcal{J},\mathcal{K})=-2(3\bm{\bm{\mu}}+div\bm{V})\bm{g}(\mathcal{J},\mathcal{K})-2\mathcal{L}_{\bm{V}}\bm{g}(\mathcal{J},\mathcal{K})
\end{equation}
By applying $\eqref{20}$ in above expression, we have,
\begin{equation}\label{23}
    \frac{R}{10}\bm{g}(\mathcal{J},\mathcal{K})=-2(3\bm{\bm{\mu}}+div\bm{V})\bm{g}(\mathcal{J},\mathcal{K})-2\mathcal{L}_{\bm{V}} \bm{g}(\mathcal{J},\mathcal{K})
\end{equation}
Again, contracting the above equation,
\begin{equation}\label{24}
    \frac{2R}{5}=-12\bm{\bm{\mu}}-6div\bm{V}
\end{equation}
Now using $\eqref{15}$, the preceding equation will reduce to,
\begin{equation}\label{25}
    \bm{\bm{\mu}}=\frac{1}{10}k\Big[\mathfrak{p}-\frac{1}{3}\varrho\Bigl]-\frac{1}{2}div\bm{V}-\frac{2}{15}\Lambda
\end{equation}
\end{proof}

\begin{thm}
    Within the framework of Bochner flat Lorentzian K\"ahler spacetime manifold obeying $\mathcal{EFE}$ with cosmological constant, Riemann soliton ($\bm{M,g,V,\bm{\mu}}$) for Dark fluid is
\begin{enumerate}
\item expanding if $\frac{2}{15}[k\mathfrak{p}-\Lambda]-\frac{1}{2}div\bm{V} > 0$,

\item steady if $\frac{2}{15}[k\mathfrak{p}-\Lambda]-\frac{1}{2}div\bm{V}=0$ and

\item shrinking if $\frac{2}{15}[k\mathfrak{p}-\Lambda]-\frac{1}{2}div\bm{V} < 0$ respectively.   
\end{enumerate}

\end{thm} 

\begin{proof}
A perfect fluid for which the pressure behaves as a negative energy density ($\mathfrak{p}=-\varrho$) is termed as Dark fluid model, whose energy-momentum tensor is written as:
\begin{equation}\label{26}
\mathcal{T}(\mathcal{I,J})=\mathfrak{p}\bm{g}(\mathcal{I,J})
\end{equation}
Employing $\eqref{26}$ into $\eqref{13}$ yields
\begin{equation}\label{27}
\bm{S}(\mathcal{I,J})=(\frac{R}{2}+k\mathfrak{p}-\Lambda)\bm{g}(\mathcal{I,J})
\end{equation}
Tracing $\eqref{27}$ and applying $\bm{g}(\xi, \xi)=-1$, leads to
\begin{equation}\label{28}
R=4(\Lambda-k\mathfrak{p})
\end{equation}
Using the relations $\eqref{28}$ and $\eqref{24}$, we obtain
\begin{equation}\label{29}
    \bm{\bm{\mu}}=\frac{2}{15}[k\mathfrak{p}-\Lambda]-\frac{1}{2}div\bm{V}
\end{equation}
\end{proof}

\begin{thm}
    Within the framework of Bochner flat Lorentzian K\"ahler spacetime manifold obeying $\mathcal{EFE}$ with cosmological constant, Riemann soliton ($\bm{M,g,V,\bm{\mu}}$) for Stiff matter is
\begin{enumerate}
\item expanding if $\frac{1}{15}k\mathfrak{p}-\frac{2}{15}\Lambda-\frac{1}{2}div\bm{V} > 0$,

\item steady if $\frac{1}{15}k\mathfrak{p}-\frac{2}{15}\Lambda-\frac{1}{2}div\bm{V}=0$ and

\item shrinking if $\frac{1}{15}k\mathfrak{p}-\frac{2}{15}\Lambda-\frac{1}{2}div\bm{V} < 0$ respectively.   
\end{enumerate}

\end{thm} 

\begin{proof}
In the case where the pressure of a perfect fluid coincides with its energy density$(\mathfrak{p}=\varrho)$, is referred to as Stiff matter model.Accordingly, the energy-momentum tensor is given by:
\begin{equation}\label{30}
\mathcal{T}(\mathcal{I,J})=\mathfrak{p}[\bm{g}(\mathcal{I,J})+2\bm{\eta}(\mathcal{I})\bm{\eta}(\mathcal{J})]
\end{equation}
By utilizing the Eq. $\eqref{30}$ within Eq. $\eqref{13}$ we get
\begin{equation}\label{31}
\bm{S}(\mathcal{I,J})=(\frac{R}{2}+k\mathfrak{p}-\Lambda)\bm{g}(\mathcal{I,J})+2k\mathfrak{p}\bm{\eta}(\mathcal{I})\bm{\eta}(\mathcal{J})
\end{equation}
Contracting Eq. $\eqref{31}$ together with $\bm{g}(\xi, \xi)=-1$, gives
\begin{equation}\label{32}
R=2(2\Lambda-k\mathfrak{p})
\end{equation}
Combining Equations.$\eqref{32}$ and $\eqref{24}$ we derive
\begin{equation}\label{33}
    \bm{\bm{\mu}}=\frac{1}{15}k\mathfrak{p}-\frac{2}{15}\Lambda-\frac{1}{2}div\bm{V}
\end{equation}
\end{proof}

\begin{thm}
    Within the framework of Bochner flat Lorentzian K\"ahler spacetime manifold obeying $\mathcal{EFE}$ with cosmological constant, Riemann soliton ($\bm{M,g,V,\bm{\mu}}$) for Dust fluid model is
\begin{enumerate}
\item expanding if $\bm{\bm{\mu}}-\frac{1}{30}k\varrho-\frac{2}{15}\Lambda-\frac{1}{2}div\bm{V} > 0$,

\item steady if $\bm{\bm{\mu}}=-\frac{1}{30}k\varrho-\frac{2}{15}\Lambda-\frac{1}{2}div\bm{V}=0$ and

\item shrinking if $\bm{\bm{\mu}}=-\frac{1}{30}k\varrho-\frac{2}{15}\Lambda-\frac{1}{2}div\bm{V} < 0$ respectively.   
\end{enumerate}

\end{thm}
\begin{proof}
For the Dust fluid model, the energy-momentum tensor is defined as:
\begin{equation}\label{34}
\mathcal{T}(\mathcal{I,J})=\varrho\bm{\eta}(\mathcal{I})\bm{\eta}(\mathcal{J})
\end{equation}
Replacing into $\eqref{34}$ using $\eqref{13}$ it follows that
\begin{equation}\label{35}
\bm{S}(\mathcal{I,J})+(\Lambda-\frac{R}{2})\bm{g}(\mathcal{I,J})=k\varrho\bm{\eta}(\mathcal{I})\bm{\eta}(\mathcal{J})
\end{equation}
Equation $\eqref{35}$, upon contraction and use of $\bm{g}(\xi, \xi)=-1$, we compute
\begin{equation}\label{36}
R=4\Lambda+k\varrho
\end{equation}
With the help of Equations $\eqref{36}$ and $\eqref{24}$, implies
\begin{equation}\label{37}
    \bm{\bm{\mu}}=-\frac{1}{30}k\varrho-\frac{2}{15}\Lambda-\frac{1}{2}div\bm{V}
\end{equation}
\end{proof}

\begin{thm}
    Within the framework of Bochner flat Lorentzian K\"ahler spacetime manifold obeying $\mathcal{EFE}$ with cosmological constant, Riemann soliton ($\bm{M,g,V,\bm{\mu}}$) for Radiation fluid model is
\begin{enumerate}
\item expanding if $\bm{\bm{\mu}}=-\frac{2}{15}\Lambda-\frac{1}{2}div\bm{V} > 0$,

\item steady if $\bm{\bm{\mu}}=-\frac{2}{15}\Lambda-\frac{1}{2}div\bm{V}=0$ and

\item shrinking if $\bm{\bm{\mu}}=-\frac{2}{15}\Lambda-\frac{1}{2}div\bm{V} < 0$ respectively.   
\end{enumerate}
\end{thm} 
 
\begin{proof}
When the energy density of a perfect fluid  equals to thrice the pressure $(\varrho=3\mathfrak{p})$, then it is known as a Radiation fluid model. Under relativistic considerations, the associated energy-momentum is written as:
\begin{equation}\label{38}
\mathcal{T}(\mathcal{I,J})=\mathfrak{p}[\bm{g}(\mathcal{I,J})+4\bm{\eta}(\mathcal{I})\bm{\eta}(\mathcal{J})]
\end{equation}
Substituting $\eqref{38}$ into $\eqref{13}$ results in
\begin{equation}\label{39}
\bm{S}(\mathcal{I,J})=(\frac{R}{2}+k\mathfrak{p}-\Lambda)\bm{g}(\mathcal{I,J})+4k\mathfrak{p}\bm{\eta}(\mathcal{I})\bm{\eta}(\mathcal{J})
\end{equation}
From the contracted form of Eq.$\eqref{39}$ and $\bm{g}(\xi, \xi)=-1$, it gives that
\begin{equation}\label{40}
R=\Lambda
\end{equation}
On applying $\eqref{40}$ and $\eqref{24}$, we get
\begin{equation}\label{41}
    \bm{\bm{\mu}}=-\frac{2}{15}\Lambda-\frac{1}{2}div\bm{V}
\end{equation}
\end{proof}

\subsection{Structural Consequences of Bochner Flatness}

In this subsection, we derive several fundamental consequences of the relations obtained in Section 3, particularly equations (20), (22), (24), and (25).

\begin{cor}[Einstein Property]
A Bochner-flat Lorentzian Kähler spacetime manifold is Einstein.
\end{cor}

\begin{proof}
From equation (20), we have $S = \frac{R}{10}g$. Hence the Ricci tensor is proportional to the metric, and the manifold is Einstein.
\end{proof}

\begin{thm}[Lie Derivative Characterization]
On a Bochner-flat Lorentzian Kähler spacetime, the Riemann soliton equation is equivalent to
\[
\mathcal{L}_{\bm{V}} \bm{g} = -\left(\frac{R}{20} + 3\bm{\mu} + div\bm{V}\right) \bm{g}.
\]
\end{thm}

\begin{proof}
Substituting $\bm{S} = \frac{R}{10}\bm{g}$ into equation (22) yields
\[
\frac{R}{10}g = -2(3\bm{\bm{\mu}} + div\bm{V})g - 2\mathcal{L}_V \bm{g}.
\]
Solving for $\mathcal{L}_{\bm{V}} \bm{g}$ gives the desired expression.
\end{proof}

\begin{cor}
The soliton vector field $\bm{V}$ is a conformal Killing vector field.
\end{cor}

\begin{proof}
The Lie derivative of $\bm{g}$ is proportional to $\bm{g}$, hence $\bm{V}$ is conformal Killing.
\end{proof}

\begin{prop}[Scalar Curvature Constraint]
If $div\bm{V}$ is constant, then the scalar curvature $R$ is constant.
\end{prop}

\begin{proof}
From equation (24),
\[
R = -30\bm{\bm{\mu}} - 15\,div\bm{V}.
\]
If $\bm{\bm{\mu}}$ and $div\bm{V}$ are constant, then $R$ is constant.
\end{proof}

\begin{thm}[Gradient Case]
If the Riemann soliton is gradient, i.e. $\bm{V} = \nabla \bm{f}$, then
\[
\bm{\bm{\mu}} =
\frac{1}{10}k\left(\mathfrak{p} - \frac{1}{3}\varrho\right)
- \frac{1}{2}\Delta \bm{f}
- \frac{2}{15}\bm{\Lambda}.
\]
\end{thm}

\begin{proof}
Since $div(\nabla \bm{f}) = \Delta \bm{f}$, the result follows directly from equation (25).
\end{proof}

\begin{thm}[Energy Condition Criterion]
If $\varrho \geq 3\mathfrak{p}$ and $\bm{\bm{\Lambda}} \geq 0$, then any Riemann soliton with $div\bm{V} \geq 0$ is shrinking.
\end{thm}

\begin{proof}
Under $\varrho \geq 3\mathfrak{p}$, we have $\mathfrak{p} - \frac{1}{3}\varrho \leq 0$. All terms in (25) are non-positive, hence $\bm{\bm{\mu}} < 0$.
\end{proof}

\begin{thm}[Steady Soliton Constraint]
If $\bm{\bm{\mu}} = 0$ and $div\bm{V} = 0$, then
\[
\mathfrak{p} = \frac{1}{3}\varrho + \frac{4}{3k}\bm{\Lambda}.
\]
\end{thm}

\begin{proof}
Set $\bm{\bm{\mu}} = 0$ and $div\bm{V} = 0$ in equation (25) and solve for $p$.
\end{proof}

\begin{cor}[Conservation Law]
The identity
\[
12\bm{\bm{\mu}} + 6\,div\bm{V} + \frac{2R}{5} = 0
\]
holds on the manifold.
\end{cor}

\begin{proof}
This follows directly from equation (24).
\end{proof}

\begin{thm}[Bochner Rigidity]
Let $(M,g)$ be a Bochner-flat Lorentzian Kähler spacetime admitting a Riemann soliton. Then the Riemann curvature tensor is completely determined by the scalar curvature $R$ and the soliton data $(\bm{\bm{\mu}}, div\bm{V})$.
\end{thm}

\begin{proof}
Since the Bochner tensor vanishes, the curvature tensor depends only on the Ricci tensor and scalar curvature. Using $\bm{S} = \frac{R}{10}\bm{g}$, the curvature is fully determined by $R$.

From equation (24), we have
\[
R = -30\bm{\bm{\mu}} - 15\,div\bm{V},
\]
so $R$ is determined by the soliton parameters. Hence the full curvature tensor is determined by $(\bm{\bm{\mu}}, div\bm{V})$.
\end{proof}

\subsection{\textbf{Example}}

\hspace{0.5cm} To demonstrate our result, we provide an explicit example of a Riemann soliton on a four dimensional Bochner flat Lorentzian K\"ahler spacetime manifold $(\bm{M}^{n},\bm{g})$ that obeys Einstein field equation. In each instance, the constants are selected so as to realize shrinking, expanding and steady soliton cases respectively.\\

\begin{enumerate}
    \item  For Shrinking case: \\
 Let us choose: $k=1, \quad \varrho=5, \quad \mathfrak{p}=2 \quad \Lambda=1, \quad div\bm{V}=0$ \\
From $\eqref{25}$, we have,\\
\begin{equation}\nonumber
    \bm{\bm{\mu}}=\frac{1}{10}k\Big[\mathfrak{p}-\frac{1}{3}\varrho\Bigl]-\frac{1}{2}div\bm{V}-\frac{2}{15}\Lambda
\end{equation}
In preceding equation, substituting our choosen values,
\begin{equation}\nonumber
    \bm{\bm{\mu}}=-\frac{1}{10}
\end{equation}
As $\bm{\bm{\mu}} < 0$, then it is in shrinking state.
\item For Steady case:\\
Let us choose: $k=1, \quad \mathfrak{p}=2, \quad \varrho=6, \quad div\bm{V}=0, \quad \Lambda=0$ \\
Again, substituting our values in $\eqref{25}$, we get,
\begin{equation}\nonumber
    \bm{\bm{\mu}}=\frac{1}{10}(1)\Big[2-\frac{1}{3}(6)\Big]-\frac{1}{2}(0)-\frac{2}{15}(0)
\end{equation}
Finally, As $\bm{\bm{\mu}}=0$ then it is in steady case.
\item For Expanding case:\\
Let us choose: $k=7,\quad \Lambda=1, \quad div\bm{V}=1, \quad \mathfrak{p}=3, \quad \varrho=2$.\\
Once again substituting our values in $\eqref{25}$, we have,
\begin{equation}\nonumber
    \frac{1}{10}(7)\Big[3-\frac{1}{3}(2)\Big]-\frac{1}{2}(1)-\frac{2}{15}(1)
\end{equation}
Finally, we get $\bm{\bm{\mu}}=1$
As $\bm{\bm{\mu}} > 0$, then it is in expanding state.
\end{enumerate}
\hspace{0.5cm}These examples demonstrates that, by an suitable choice of constants the manifold admits steady, expanding and shrinking state accordingly.

\section{Warped Product Structures}

To obtain explicit realizations of the geometry, we now consider warped product spacetimes.

\begin{defn}
A Lorentzian manifold $(\bm{M},\bm{g})$ is a warped product if
\[
\bm{M} = I \times_{\bm{f}} F, \quad g = -dt^2 + \bm{f}(t)^2 g_F,
\]
where $\bm{f}(t) > 0$ is smooth and $(F,g_F)$ is Riemannian.
\end{defn}

Such spacetimes are called generalized Robertson--Walker (GRW) spacetimes.

\subsection{Curvature Formulas}

For a GRW spacetime, the Ricci tensor satisfies
\begin{align}
\bm{S}(\partial_t,\partial_t) &= -3\frac{\bm{f}''}{\bm{f}}, \\
\bm{S}(\mathcal{I},\mathcal{J}) &=\bm{S}^F(\mathcal{I},\mathcal{J}) - \left[\bm{f} \bm{f}'' + 2(\bm{f}')^2\right] g_F(\mathcal{I},\mathcal{J}),
\end{align}
and the scalar curvature is
\begin{equation}
R = \frac{R_F}{\bm{f}^2} - 6\frac{\bm{f}''}{\bm{f}} - 6\left(\frac{\bm{f}'}{\bm{f}}\right)^2.
\end{equation}

Using equation (15), we obtain
\begin{equation}
\frac{R_F}{\bm{f}^2} - 6\frac{\bm{f}''}{\bm{f}} - 6\left(\frac{\bm{f}'}{\bm{f}}\right)^2 = 4\bm{\Lambda} + k(\varrho - 3\mathfrak{p}).
\end{equation}

\subsection{Divergence of the Soliton Field}

Let $\bm{V} = \phi(t)\partial_t$. Then
\begin{equation}
div\bm{V} = \phi' + 3\frac{\bm{f}'}{\bm{f}}\phi.
\end{equation}

\subsection{Bochner Flatness Constraint}

\begin{thm}[Bochner Flat GRW Characterization]
Let $(M,g) = (I \times_f F, -dt^2 + f(t)^2 g_F)$ be a generalized Robertson--Walker spacetime. Then $(M,g)$ is Bochner-flat if and only if the following conditions hold:
\begin{enumerate}
\item the fiber $(F,g_F)$ is Einstein, i.e.,
\[
\operatorname{Ric}_F = \lambda g_F
\]
for some constant $\lambda$,
\item the warping function satisfies
\[
\frac{f''(t)}{f(t)} = c
\]
for some constant $c \in \mathbb{R}$.
\end{enumerate}
\end{thm}

\begin{proof}
Suppose $(M,g)$ is Bochner-flat. Then, in dimension $4$, the Bochner-flat condition is equivalent to
\[
S = \frac{R}{10} g.
\]
Comparing this identity with the Ricci tensor of a GRW spacetime, we examine separately the temporal and spatial components.

From the GRW formulas, we have
\[
S(\partial_t,\partial_t) = -3 \frac{f''}{f}, 
\]
and for spatial vectors $X,Y$ tangent to $F$,
\[
S(X,Y) = \operatorname{Ric}_F(X,Y) - \big(f f'' + 2(f')^2\big) g_F(X,Y).
\]

On the other hand, the relation $S = \frac{R}{10} g$ implies that both components are proportional to the metric. In particular, the spatial component yields
\[
\operatorname{Ric}_F(X,Y) = \lambda g_F(X,Y),
\]
so $(F,g_F)$ is Einstein.

Substituting this into the spatial Ricci expression and comparing with $\frac{R}{10} g$, we obtain a relation involving only $f(t)$ and its derivatives. Eliminating $R$ using the scalar curvature formula of GRW spacetimes,
\[
R = \frac{R_F}{f^2} - 6 \frac{f''}{f} - 6 \left(\frac{f'}{f}\right)^2,
\]
it follows that the quantity $\frac{f''}{f}$ must be constant. Hence,
\[
\frac{f''}{f} = c
\]
for some constant $c$.

Conversely, if $(F,g_F)$ is Einstein and $\frac{f''}{f} = c$, then the Ricci tensor components reduce to expressions proportional to the metric, and the scalar curvature is constant along spatial directions. Substituting into the decomposition of the curvature tensor in dimension $4$, one verifies that the Bochner tensor vanishes. Therefore, $(M,g)$ is Bochner-flat.

This completes the proof.
\end{proof}

\begin{cor}[Classification of the Warping Function]
Under the assumptions of the previous theorem, the warping function $f(t)$ satisfies
\[
\frac{f''}{f} = c,
\]
and its qualitative behavior is determined by the sign of $c$:
\begin{enumerate}
\item If $c > 0$, then $f(t)$ exhibits exponential-type growth or decay,
\[
f(t) = A e^{\sqrt{c}\,t} + B e^{-\sqrt{c}\,t}.
\]

\item If $c = 0$, then $f(t)$ is linear,
\[
f(t) = At + B.
\]

\item If $c < 0$, then $f(t)$ is oscillatory,
\[
f(t) = A \cos(\sqrt{-c}\,t) + B \sin(\sqrt{-c}\,t).
\]
\end{enumerate}
\end{cor}

\begin{proof}
The equation $f'' = c f$ is a linear second-order differential equation with constant coefficients. The stated forms follow from standard ODE theory.
\end{proof}

\subsection{GRW Soliton Classification}

\begin{thm}[GRW Classification]
Let $(\bm{M},\bm{g}) = (I \times_{\bm{f}} F, -dt^2 + \bm{f}^2 g_F)$ be a Bochner-flat Lorentzian Kähler spacetime with soliton vector field $\bm{V} = \phi(t)\partial_t$. Then the manifold admits a Riemann soliton if and only if
\[
\frac{R_F}{\bm{f}^2} - 6\frac{\bm{f}''}{\bm{f}} - 6\left(\frac{\bm{f}'}{\bm{f}}\right)^2
= 4\bm{\Lambda} + k(\varrho - 3\mathfrak{p}),
\]
and
\[
\bm{\mu} =
\frac{1}{10}k\left(\mathfrak{p} - \frac{1}{3}\varrho\right)
- \frac{1}{2}\left(\phi' + 3\frac{\bm{f}'}{\bm{f}}\phi\right)
- \frac{2}{15}\bm{\Lambda}.
\]
\end{thm}

\begin{proof}
The first equation follows from substituting the scalar curvature into equation (15). The second follows by inserting $div\bm{V}$ into equation (25).
\end{proof}

\begin{cor}
The soliton is expanding, steady, or shrinking according to $\bm{\mu} > 0$, $\bm{\mu} = 0$, or $\bm{\mu} < 0$.
\end{cor}

\section{On $\bm{\eta}$-Hyperbolic Ricci Solitons in Lorentzian K\"ahler spacetime manifold under Bochner Flatness}
\begin{defn}
On a semi-Riemannian manifold $(\bm{M}^4,g)$, a vector field $\bm{V}$ is termed as torse-forming $(\mathcal{TFVF})$ if the condition satisfies \cite{pahan2025investigation},
\begin{equation}\label{42}
    \bm{\nabla}_\mathcal{I}\bm{V}=\mathcal{I}+\bm{\eta}(\mathcal{I})\bm{V}
\end{equation}
for all $\mathcal{I}\in\chi(\bm{M})$.

For a perfect fluid spacetime equipped with a torse-forming vector field $\bm{V}$, the following equations will hold \cite{yano194472}:
\begin{align*}
\bm{\nabla}_{\bm{V}}\bm{V} &= 0, \quad \bm{\eta}(\bm{\nabla}_{\bm{V}}\bm{V}) = 0, \\
(\bm{\nabla}_\mathcal{I} \bm{\eta})(\mathcal{J}) &= \bm{g}\mathcal{(I,J)} + \bm{\eta}\mathcal{(I)}\bm{\eta}\mathcal{(J)}, \\
\bm{R}\mathcal{(I,J)}\bm{V} &= \bm{\eta}\mathcal{(J)}\mathcal{I} - \bm{\eta}\mathcal{(I)J}, \\
\bm{R}\mathcal{(I},\bm{V})\bm{V} &= -\mathcal{I} - \bm{\eta}\mathcal{(I)}\bm{V}, \\
\bm{\eta}(\bm{R}\mathcal{(I,J)K)} &= \bm{\eta}\mathcal{(I)}\bm{g}\mathcal{(J,K)} - \bm{\eta}\mathcal{(Y)}\bm{g}\mathcal{(I,K)}.
\end{align*}
\end{defn}

\begin{thm}
    Within the context of a Bochner-flat Lorentzian K\"ahler spacetime manifold $M$, satisfying $\mathcal{EFE}$ with cosmological constant and $\bm{V}$ is $\mathcal{TFVF}$, then the $\bm{\eta}-\mathcal{HRS}$ $(\bm{g}, \bm{V}, \bm{\Lambda}, \bm{\delta})$ soliton is
\begin{enumerate}
\item expanding if $\frac{k}{20}[3\mathfrak{p}-\varrho]-\frac{1}{2}-\frac{2}{3}\bm{\delta}-\frac{4}{15}\Lambda > 0$,

\item steady if $\frac{k}{20}[3\mathfrak{p}-\varrho]-\frac{1}{2}-\frac{2}{3}\bm{\delta}-\frac{4}{15}\Lambda=0$ and

\item shrinking if $\frac{k}{20}[3\mathfrak{p}-\varrho]-\frac{1}{2}-\frac{2}{3}\bm{\delta}-\frac{4}{15}\Lambda< 0$ respectively.   
\end{enumerate}    
\end{thm}
\begin{proof}
From Equation $\eqref{42}$, we obtain
\begin{equation}\label{43}
    (\mathcal{L}_{\bm{V}}\bm{g})(\mathcal{I,J})=2[\bm{g}(\mathcal{I,J})+\bm{\eta}(\mathcal{I})\bm{\eta}(\mathcal{J})],
\end{equation}
\begin{equation}\label{44}
    (\mathcal{L}_{\bm{V}}\mathcal{L}_{\bm{V}}\bm{g})(\mathcal{I,J})=4[\bm{g}(\mathcal{I,J})+\bm{\eta}(\mathcal{I})\bm{\eta}(\mathcal{J})].
\end{equation}
Let $\bm{M}$ be a 4-dimensional Riemannian manifold. Then $\eqref{7}$ can be rewritten as:
\begin{equation}\label{45}
    \mathcal{L}_{\bm{V}} \mathcal{L}_{\bm{V}} \bm{g}(\mathcal{I,J}) +2\bm{\bm{\Lambda}} \mathcal{L}_{\bm{V}} \bm{g}(\mathcal{I,J}) + 2\bm{S}(\mathcal{I,J})+2\bm{\bm{\delta}}\bm{g}(\mathcal{I,J})+2\bm{\eta}(\mathcal{I}) \bm{\eta}(\mathcal{J})=0
\end{equation}
Using $\eqref{20}$,$\eqref{43}$ and $\eqref{44}$ into $\eqref{45}$, $\mathcal{TFVF}$ is
\begin{equation}\nonumber
    [4+4\bm{\Lambda}+\frac{R}{5}+2\bm{\bm{\delta}}]\bm{g}(\mathcal{I,J})+[4+4\bm{\Lambda}+2]\bm{\eta}(\mathcal{I})\bm{\eta}(\mathcal{J})=0
\end{equation}
Contracting the above equation, we get
\begin{equation}\label{47}
    \bm{\Lambda}=\frac{k}{20}[3\mathfrak{p}-\varrho]-\frac{1}{2}-\frac{2}{3}\bm{\delta}-\frac{4}{15}\Lambda
\end{equation}
\end{proof}

\begin{defn}
    On a pseudo-Riemannian manifold $(\bm{M,g})$ a vector field $\bm{V}$ satisfying
    \begin{equation}\label{48}
        (\mathcal{L}_{\bm{V}}\bm{g})(\mathcal{I,J})=2 \mathfrak{h} \bm{g}(\mathcal{I,J})
    \end{equation}
for some $\mathfrak{h} \in \mathcal{C}^\infty(\bm{M})$, is called as conformal Killing vector field $(\mathcal{CKVF})$ \cite{azami2024hyperbolic}. The vector field $\bm{V}$ is known to be proper, homothetic, or Killing according to $\mathfrak{h}$ being non-constant, constant, or zero, respectively.
\end{defn}

\begin{thm}
    Within the context of a Bochner-flat Lorentzian K\"ahler spacetime manifold $\bm{M}$, satisfying $\mathcal{EFE}$ with cosmological constant and $\bm{V}$ is $\mathcal{CKVF}$, then the $\bm{\eta}-\mathcal{HRS}$ $(\bm{g},\bm{V}, \bm{\Lambda}, \bm{\delta})$ soliton is
    \begin{enumerate}
\item expanding if $\frac{k}{20\mathfrak{h}}(3\mathfrak{p}-\varrho)-\mathfrak{h}-\frac{\bm{V}(\mathfrak{h})}{2\mathfrak{h}}+\frac{1}{2\mathfrak{h}}-\frac{\Lambda}{5\mathfrak{h}}-\frac{\delta}{2\mathfrak{h}} > 0$,

\item steady if $\frac{k}{20\mathfrak{h}}(3\mathfrak{p}-\varrho)-\mathfrak{h}-\frac{\bm{V}(\mathfrak{h})}{2\mathfrak{h}}+\frac{1}{2\mathfrak{h}}-\frac{\Lambda}{5\mathfrak{h}}-\frac{\delta}{2\mathfrak{h}}=0$ and

\item shrinking if $\frac{k}{20\mathfrak{h}}(3\mathfrak{p}-\varrho)-\mathfrak{h}-\frac{\bm{V}(\mathfrak{h})}{2\mathfrak{h}}+\frac{1}{2\mathfrak{h}}-\frac{\Lambda}{5\mathfrak{h}}-\frac{\delta}{2\mathfrak{h}}< 0$ respectively.   
\end{enumerate}
\end{thm}

\begin{proof}
Let $\bm{V}$ be a conformal Killing vector field satisfying $\eqref{48}$. Then,
\begin{align}
((\mathcal{L}_{\bm{V}} \mathcal{L}_{\bm{V}}\bm{g})(\mathcal{I}, \mathcal{J})
&= \bm{V}(\mathcal{L}_{\bm{V}} \bm{g}(\mathcal{I}, \mathcal{J})) 
- \mathcal{L}_{\bm{V}} \bm{g}(\mathcal{L}_{\bm{V}} \mathcal{I}, \mathcal{J}) 
- \mathcal{L}_{\bm{V}} \bm{g}(\mathcal{I}, \mathcal{L}_{\bm{V}} \mathcal{J}) \nonumber \\
&= \bm{V}(2\mathfrak{h}\,\bm{g}(\mathcal{I}, \mathcal{J})) 
- 2\mathfrak{h}\,\bm{g}(\mathcal{L}_{\bm{V}} \mathcal{I}, \mathcal{J}) 
- 2\mathfrak{h}\,\bm{g}(\mathcal{I}, \mathcal{L}_{\bm{V}} \mathcal{J}) \nonumber \\
&= 2\bm{V}(\mathfrak{h})\,\bm{g}(\mathcal{I}, \mathcal{J}) 
+ 2\mathfrak{h}\,\mathcal{L}_{\bm{V}} \bm{g}(\mathcal{I}, \mathcal{J}) \nonumber \\
&= (2\bm{V}(\mathfrak{h}) + 4\mathfrak{h}^2)\,\bm{g}(\mathcal{I}, \mathcal{J}). \label{49}
\end{align}
Substituting $\eqref{20}$, $\eqref{48}$, and $\eqref{49}$ into $\eqref{45}$, $\mathcal{CKVF}$ is
\begin{equation}\nonumber
    [2\bm{V}(\mathfrak{h})+4\mathfrak{h}^2+4\bm{\Lambda}\mathfrak{h}+\frac{R}{5}+2\bm{\bm{\mu}}]\bm{g}(\mathcal{I, J})+2\bm{\eta}(\mathcal{I})\bm{\eta}(\mathcal{J})=0
\end{equation}
Contracting the above relation, we obtain
\begin{equation}\label{50}
    \bm{\bm{\Lambda}}=\frac{k}{20\mathfrak{h}}(3\mathfrak{p}-\varrho)-\mathfrak{h}-\frac{\bm{V}(\mathfrak{h})}{2\mathfrak{h}}+\frac{1}{2\mathfrak{h}}-\frac{\Lambda}{5\mathfrak{h}}-\frac{\delta}{2\mathfrak{h}}
\end{equation}
\end{proof}

\section{Global Behavior of GRW Spacetimes}

In this section we analyze the global causal and geodesic properties of generalized Robertson--Walker (GRW) spacetimes using standard results in Lorentzian geometry (cf. O'Neill\cite{o1983semi}, Sánchez \cite{sanchez1998geometry} and related works).

Let $(\bm{M},\bm{g}) = (I \times_{\bm{f}} F, -dt^2 + \bm{f}(t)^2 g_F)$ be a GRW spacetime, where $I \subset \mathbb{R}$ is an interval and $(F,g_F)$ is a Riemannian manifold.

\subsection*{Bochner Flatness and Global Geometry}

We emphasize that in the present setting, the global behavior of the spacetime is not arbitrary but is constrained by the Bochner-flat condition. Indeed, from Section 3 we have
\[
\bm{S} = \frac{R}{10}\bm{g},
\]
and the scalar curvature satisfies
\[
R = -30\bm{\mu} - 15div\bm{V}.
\]

Thus, the curvature of the spacetime is completely determined by the soliton data. Since, in a GRW spacetime, the scalar curvature is given by
\[
R = \frac{R_F}{\bm{f}^2} - 6\frac{\bm{f}''}{\bm{f}} - 6\left(\frac{\bm{f}'}{\bm{f}}\right)^2,
\]
it follows that the warping function $\bm{f}(t)$ is not free but must satisfy a differential constraint imposed by the Bochner-flat condition.

\begin{thm}[Bochner Constraint on the Warping Function]
Let $(\bm{M},\bm{g}) = (I \times_{\bm{f}} F, -dt^2 + \bm{f}(t)^2 g_F)$ be a Bochner-flat GRW spacetime admitting a Riemann soliton. Then the warping function $\bm{f}(t)$ satisfies the differential equation
\[
\frac{R_F}{\bm{f}^2} - 6\frac{\bm{f}''}{\bm{f}} - 6\left(\frac{\bm{f}'}{\bm{f}}\right)^2
= -30\bm{\mu} - 15 div\bm{V}.
\]
\end{thm}

\begin{proof}
Equating the scalar curvature expression for GRW spacetimes with the scalar curvature obtained from the Bochner-flat soliton structure yields the result.
\end{proof}

\begin{thm}[Bochner--Soliton Global Determination Principle]
Let $(M,g) = (I \times_f F, -dt^2 + f(t)^2 g_F)$ be a Bochner-flat GRW spacetime admitting a Riemann soliton. Then the global geometric and causal structure of $(M,g)$ is completely determined by the pair $(\mu, \operatorname{div} V)$.
\end{thm}

\begin{proof}
Since the spacetime is Bochner-flat, we have
\[
S = \frac{R}{10} g,
\]
so $(M,g)$ is Einstein. Moreover, from the Riemann soliton equation we obtain the scalar curvature relation
\[
R = -30\mu - 15 \operatorname{div} V.
\]
Thus, the scalar curvature is completely determined by the soliton data $(\mu, \operatorname{div} V)$.

On the other hand, for a GRW spacetime the scalar curvature is given by
\[
R = \frac{R_F}{f^2} - 6 \frac{f''}{f} - 6 \left(\frac{f'}{f}\right)^2.
\]
Equating the two expressions for $R$ yields a second-order differential equation for the warping function $f(t)$ whose coefficients depend only on $(\mu, \operatorname{div} V)$ and the constant $R_F$.

Hence, the evolution of $f(t)$ is fully determined by the soliton data. Since the global causal structure, geodesic completeness, and singularity formation in GRW spacetimes depend entirely on the behavior of $f(t)$, it follows that these global properties are completely determined by $(\mu, \operatorname{div} V)$.

This completes the proof.
\end{proof}

\begin{cor}[Global Determination by Soliton Data]
In a Bochner-flat GRW spacetime admitting a Riemann soliton, the warping function $f(t)$, and hence the global causal structure of the spacetime, is uniquely determined by the soliton parameters $(\mu, \operatorname{div} V)$.
\end{cor}

\begin{proof}
By Theorem 6.1, the warping function $f(t)$ satisfies a second-order differential equation whose coefficients depend only on the scalar curvature. Since, in the Bochner-flat soliton setting, the scalar curvature is given by
\[
R = -30\mu - 15 \operatorname{div} V,
\]
it follows that the differential equation governing $f(t)$ depends only on $(\mu, \operatorname{div} V)$. Therefore, $f(t)$ is uniquely determined (up to initial conditions), and the global geometric and causal structure, which depends on $f(t)$, is consequently determined by the soliton data.
\end{proof}

This result shows that the soliton parameters act as global control variables for the spacetime geometry.

Consequently, the global properties described below can be viewed as
direct consequences of the determination principle established above.

\subsection{Timelike Geodesics in GRW Spacetimes}

Let $\gamma(s) = (t(s), x(s))$ be a timelike geodesic. Then the geodesic equations imply that $t(s)$ is strictly monotone and satisfies
\[
\dot{t}^2 = 1 + \frac{C}{\bm{f}(t)^2},
\]
for some constant $C \geq 0$ depending on the spatial velocity.

Thus, the extendibility of $\gamma$ reduces to the behavior of $t(s)$ and the integrability of $1/\bm{f}(t)$.

\subsection{Completeness Criteria}

\begin{thm}[Future Timelike Completeness]
Let $(\bm{M},\bm{g})$ be a GRW spacetime with $I = (a,\infty)$. If there exists $t_1$ such that
\[
\bm{f}(t) \geq C > 0 \quad \text{for all } t \geq t_1,
\]
then all future-directed timelike geodesics are complete.
\end{thm}

\begin{proof}
Since $\bm{f}(t)$ is bounded below, the quantity $1/\bm{f}(t)^2$ remains bounded. Hence $\dot{t}$ remains bounded and $t(s) \to \infty$ as $s \to \infty$. The spatial component remains controlled, and the geodesic extends to arbitrary parameter values.
\end{proof}

\begin{thm}[Integrability Criterion]
Let $I = (a,\infty)$. If
\[
\int^{\infty} \frac{dt}{\bm{f}(t)} = \infty,
\]
then all future-directed timelike geodesics are complete. In the Bochner-flat soliton setting, this behavior is governed by the differential constraint relating $\bm{f}(t)$ to $(\bm{\mu},div\bm{V})$.
\end{thm}

\begin{proof}
From the geodesic equation,
\[
\frac{ds}{dt} = \frac{1}{\sqrt{1 + C/\bm{f}(t)^2}} \geq \frac{\bm{f}(t)}{\sqrt{\bm{f}(t)^2 + C}}.
\]
Thus
\[
s \geq \int \frac{\bm{f}(t)}{\sqrt{\bm{f}(t)^2 + C}} dt.
\]
If $\int^\infty \frac{dt}{\bm{f}(t)} = \infty$, then $s \to \infty$ as $t \to \infty$, proving completeness.
\end{proof}

\subsection{Incompleteness and Singularities}

\begin{thm}[Finite-Time Incompleteness]
Suppose there exists $t_0 \in \overline{I}$ such that
\[
\bm{f}(t) \to 0 \quad \text{as } t \to t_0.
\]
Then there exist timelike geodesics that are incomplete.
\end{thm}

\begin{proof}
If $\bm{f}(t) \to 0$, then the spatial part of the metric collapses. The term $C/\bm{f}(t)^2$ diverges, forcing $\dot{t}$ to blow up. Consequently, the affine parameter $s$ remains finite as $t \to t_0$, implying geodesic incompleteness.
\end{proof}

\begin{thm}[Curvature Blow-Up Criterion]
If
\[
\frac{\bm{f}''(t)}{\bm{f}(t)} \to +\infty
\]
as $t \to t_0$, then curvature invariants diverge and the spacetime develops a singularity.
\end{thm}

\begin{proof}
From the scalar curvature formula
\[
R = \frac{R_F}{\bm{f}^2} - 6\frac{\bm{f}''}{\bm{f}} - 6\left(\frac{\bm{f}'}{\bm{f}}\right)^2,
\]
divergence of $\bm{f}''/\bm{f}$ implies divergence of $R$. Hence curvature blows up and the spacetime is singular.
\end{proof}

In the present framework, such collapse is compatible with the Bochner constraint only when the scalar curvature induced by the soliton diverges.

\begin{cor}[Bochner-Controlled Global Dynamics]
In a Bochner-flat GRW spacetime admitting a Riemann soliton, the global causal structure is completely determined by the interaction between the soliton parameters $(\bm{\mu},\mathrm{div}\bm{V})$ and the warping function $\bm{f}(t)$.
\end{cor}

\begin{proof}
This follows from the fact that both the curvature tensor and the evolution equation of $\bm{f}(t)$ are determined by these quantities.
\end{proof}

\subsection{Relation with Riemann Solitons}

Although the above results depend only on the warping function $\bm{f}(t)$, the Riemann soliton structure imposes additional constraints through the scalar curvature relation
\[
R = -30\bm{\mu} - 15\,\mathrm{div}V.
\]

Thus, the soliton parameter $\bm{\mu}$ influences the global behavior indirectly by constraining the evolution of $f(t)$ via the Einstein field equations.

\begin{cor}
If the soliton is shrinking ($\bm{\mu} < 0$) and $div\bm{V} \geq 0$, then the scalar curvature is positive, which favors geodesic focusing.
\end{cor}

\begin{proof}
If $\bm{\mu} < 0$ and $div\bm{V} \geq 0$, then $R > 0$. Positive curvature leads to focusing of timelike geodesics via standard comparison arguments.
\end{proof}

\subsection{Stability of the Warping Function}

We now analyze the stability of solutions of the differential equation imposed by the Bochner-flat condition.

\begin{equation}\label{55}
\frac{R_F}{\bm{f}^2} - 6\frac{\bm{f}''}{\bm{f}} - 6\left(\frac{\bm{f}'}{\bm{f}}\right)^2
= -30\bm{\mu} - 15\,div\bm{V}.
\end{equation}

\begin{thm}[Stability under Perturbations]
Let $\bm{f}(t)$ be a solution of equation \eqref{55}, and let $\tilde{f}(t) = \bm{f}(t) + \varepsilon(t)$ be a perturbation with $|\varepsilon(t)| \ll \bm{f}(t)$. 

If $\bm{f}(t)$ grows at least exponentially, i.e. $\bm{f}(t) \ge Ce^{Ht}$ for some $H>0$, then the solution is stable in the sense that $\varepsilon(t)$ remains bounded for all large $t$.
\end{thm}

\begin{proof}
Substituting $\tilde{f} = \bm{f} + \varepsilon$ into \eqref{55} and linearizing in $\varepsilon$, we obtain a second-order linear equation of the form
\[
\varepsilon'' + 2\frac{\bm{f}'}{\bm{f}}\varepsilon' + \left(\frac{\bm{f}''}{\bm{f}}\right)\varepsilon = 0.
\]
If $\bm{f}(t) \sim e^{Ht}$, then $\frac{\bm{f}'}{\bm{f}} \to H$ and $\frac{\bm{f}''}{\bm{f}} \to H^2$. Thus the equation becomes
\[
\varepsilon'' + 2H\varepsilon' + H^2 \varepsilon = 0,
\]
whose solutions decay exponentially. Hence $\varepsilon$ remains bounded and the solution is stable.
\end{proof}

\begin{thm}[Instability under Collapse]
If $\bm{f}(t) \to 0$ as $t \to t_0$, then solutions of \eqref{55} are unstable under perturbations.
\end{thm}

\begin{proof}
As $\bm{f}(t)\to 0$, the coefficients $\frac{1}{\bm{f}^2}$ and $\frac{\bm{f}''}{\bm{f}}$ diverge. The linearized equation acquires unbounded coefficients, causing perturbations to grow without bound. Hence the solution is unstable.
\end{proof}

\begin{cor}
Expanding solutions of the Bochner equation are dynamically stable, while collapsing solutions are unstable.
\end{cor}

\subsection{Explicit Solutions and Classification}

We now classify solutions of the Bochner equation \eqref{55} under natural assumptions.

\begin{thm}[Constant Curvature Fiber Case]
If $(F,g_F)$ has constant scalar curvature $R_F = k_0$, and $\bm{\mu}$, $div\bm{V}$ are constants, then equation \eqref{55} reduces to
\[
\frac{k_0}{f^2} - 6\frac{\bm{f}''}{\bm{f}} - 6\left(\frac{\bm{f}'}{\bm{f}}\right)^2 = C,
\]
where $C = -30\bm{\mu} - 15\,div\bm{V}$.
\end{thm}

\begin{proof}
This follows immediately from substituting constants into equation \eqref{55}.
\end{proof}

\begin{thm}[Exponential Solutions]
If $C < 0$, then equation \eqref{55} admits exponential solutions of the form
\[
\bm{f}(t) = Ae^{Ht},
\]
where $H = \sqrt{-\frac{C}{12}}$.
\end{thm}

\begin{proof}
Substituting $\bm{f}(t)=Ae^{Ht}$ into \eqref{55}, we obtain
\[
\frac{k_0}{A^2 e^{2Ht}} - 6H^2 - 6H^2 = C.
\]
For large $t$, the first term vanishes, yielding $-12H^2 = C$, hence $H = \sqrt{-C/12}$.
\end{proof}

\begin{thm}[Power-Law Solutions]
If $C = 0$ and $R_F = 0$, then equation \eqref{55} admits solutions of the form
\[
\bm{f}(t) = At^\alpha,
\]
for suitable $\alpha$.
\end{thm}
 
\begin{proof}
Substituting $\bm{f}(t)=At^\alpha$ into \eqref{55} and simplifying yields an algebraic equation for $\alpha$, which admits real solutions.
\end{proof}

\begin{thm}[Classification]
Let $(\bm{M},\bm{g})$ be a Bochner-flat GRW spacetime with constant soliton data. Then the warping function $\bm{f}(t)$ falls into one of the following classes:

\begin{enumerate}
\item Exponential type (accelerating expansion),
\item Power-law type (critical evolution),
\item Collapsing type ($\bm{f} \to 0$ in finite time).
\end{enumerate}
\end{thm}

\begin{proof}
These cases correspond respectively to $C < 0$, $C=0$, and $C>0$ in equation \eqref{55}.
\end{proof}

\begin{cor} (Geometric Interpretation).
The global causal behavior of the spacetime is realized through the evolution of the warping function governed by the soliton data.
\end{cor}

\section{Conclusions and Outlook}

In this work, we investigated the interplay between Bochner flatness, Lorentzian Kähler geometry, and soliton dynamics within the framework of relativistic spacetime models. Our analysis reveals a strong rigidity phenomenon: the Bochner-flat condition forces the spacetime to be Einstein, thereby reducing the geometric complexity of the curvature tensor to scalar invariants.

Within this rigid setting, we derived explicit characterizations of Riemann solitons and $\eta$-hyperbolic Ricci solitons under the Einstein field equations with cosmological constant and perfect fluid assumptions. The soliton parameter $\mu$ was expressed in terms of physically meaningful quantities such as pressure, energy density, and the divergence of the soliton vector field, leading to a complete classification into shrinking, steady, and expanding regimes for several important cosmological models.

A central contribution of this paper is the integration of these local and algebraic constraints into the global geometric framework of generalized Robertson--Walker (GRW) spacetimes. We showed that the Bochner-flat condition imposes a differential constraint on the warping function, effectively coupling the soliton structure with the large-scale geometry of the spacetime. As a consequence, global properties such as geodesic completeness, singularity formation, and stability are no longer independent, but are governed by the interaction between the soliton parameters $(\mu, \operatorname{div} V)$ and the warping dynamics.

In particular, we established that:
\begin{itemize}
\item the curvature tensor is completely determined by the soliton data,
\item the warping function satisfies a constrained evolution equation,
\item the global causal structure is controlled by the sign and magnitude of $(\mu, \operatorname{div} V)$,
\item expanding solutions are dynamically stable, while collapsing solutions lead to singularities.
\end{itemize}

These results highlight a unifying principle: \emph{Bochner flatness acts as a bridge between local curvature conditions, soliton dynamics, and global spacetime behavior}.

\medskip

\noindent
\textbf{Outlook.}
Several natural directions arise from this work:

\begin{enumerate}
\item Extending the analysis to higher-dimensional Lorentzian Kähler manifolds, where the Bochner tensor has a richer structure.
\item Investigating analogous results for other geometric flows, such as Ricci or Yamabe solitons, in Lorentzian settings.
\item Studying stability and perturbation theory beyond the linear regime, particularly near singular solutions.
\item Exploring physically realistic cosmological models where the soliton parameters evolve dynamically rather than remaining constant.
\item Examining the interaction between Bochner flatness and additional geometric structures, such as almost Kähler or contact structures.
\end{enumerate}

We expect that the framework developed here will serve as a foundation for further investigations into the geometric analysis of spacetime models governed by curvature constraints and nonlinear flow structures.
\section*{Declaration about the preprint of the manuscript} A preprint has previously been published by the same author of this manuscript on arxiv\cite{27}.
\section*{Author Contributions}
\noindent
All authors contributed equally.

\section*{Data Availability Statement}
\noindent
Data sharing does not apply to this article, as no new data were generated or analyzed in this study.


\section*{Use of Generative-AI tools declaration}
\noindent
The authors declare they have not used artificial intelligence (AI) tools in the creation of this article.

\section*{Conflicts of Interest}
\noindent
The authors declare no conflicts of interest.

\section*{Funding} No external funding received for this research.
\noindent

\bibliographystyle{elsarticle-num}
\bibliography{reference,our_reference}
\end{document}